\documentclass[twocolumn,11pt]{article}
\usepackage{times}
\usepackage{graphicx}
\usepackage{amsmath}
\usepackage{amsfonts}

\newtheorem{definition}{Definition}

\newtheorem{theorem}[definition]{Theorem}

\begin{document}

\title{{\LARGE \textbf{Homogenization of a non standard transmission problem}%
}}
\date{}
\author{\hspace*{-10pt}
\begin{minipage}[t]{2.7in} \normalsize \baselineskip 12.5pt
\centerline{ABDELHAMID AINOUZ}
\centerline{Laboratory AMNEDP, Faculty of Mathematics}
\centerline{University of sciences
and technology, Houari Boumedienne}
\centerline{Po Box 32 El Alia, Algiers}
\centerline{ALGERIA}
\centerline{aainouz@usthb.dz}
\end{minipage} \kern 0in \\
\\
\hspace*{-10pt}
\begin{minipage}[b]{6.9in} \normalsize
\baselineskip 12.5pt {\it Abstract:}
In this paper, periodic homogenization of a steady heat flow in
two-component media with highly adhesive contact is performed via the
two-scale convergence technique. Our micro-model is based on mass
conservation for the heat flow in each phase with interfacial contact of
adhesive type between these constituents. It is shown that the macroscopic
model is a single phase elliptic equation.
\\ [4mm] {\it Key--Words:}
Homogenization, Two-scale convergence, non standard transmission conditions
\end{minipage}
\vspace{-10pt}}
\maketitle

\baselineskip 12.5pt %
%
%

\thispagestyle{empty} \pagestyle{empty}
%
%

\section{Introduction\label{s1}}

Diffusion processes in multi-component media with non standard transmission
conditions play an important role in many areas of mechanical engineering
such as reservoir petroleum, biomechanics, geophysics,... In this short
paper, we shall deal with the homogenization of a steady heat flow in media
made of two interacting systems with an interfacial barrier leading to a
jump of heat flux (see for instance \cite{c22drm, c22drm2, c22mis, c22mis1,
c22mis2, c22rc} and the references therein). Our micro-model is a two-phase
elliptic system consisting of two equations as a result of two different
components and based on mass conservation for the heat flow in each phase,
combined with the Fick's law and non standard transmission interfacial
conditions between these constituents. The macro-model is derived by means
of the two-scale convergence method \cite{all}. It is shown that the overall
behavior of the heat diffusion process such media obeys to a single phase
equation.

The paper is organized as follows. In section \ref{c22s2}, we set our
microscopic model. We also give an existence and uniqueness result. At the
end of this section we state the main result of this paper: Theorem \ref%
{c22thp}. Section \ref{c22s3} is devoted to the proof of this Theorem.

\section{Setting of the Problem and the main result\label{c22s2}}

We consider $\Omega $\ a bounded and smooth domain of $\mathbb{R}^{N}\ $($%
N\geq 2$) and $Y=]0,1[^{N}$\ the generic cell of periodicity. Let $%
Y_{1},Y_{2}\subset Y$ be\ two open disjoint subsets of $Y$ such that $%
Y=Y_{1}\cup Y_{2}\cup \Sigma $ where $\Sigma =\partial Y_{1}\cap \partial
Y_{2}$, assumed to be a smooth submanifold. We denote $\nu $ the unit normal
of $\Sigma $, outward to $Y_{1}$. For $i=1,2,$ let $\chi _{i}$ denote the
characteristic function of $Y_{i}$, extended by $Y$-periodicity to $\mathbb{R%
}^{N}$. For $\varepsilon >0$, we set%
\begin{equation*}
\Omega _{i}^{\varepsilon }=\{x\in \Omega :\chi _{i}(\frac{x}{\varepsilon }%
)=1\}\ \ \ i=1,2
\end{equation*}%
and\ $\Sigma ^{\varepsilon }=\partial \Omega _{1}^{\varepsilon }\cap
\partial \Omega _{2}^{\varepsilon }$. To avoid some\ unnecessary technical
computations, we assume that $\overline{\Omega _{2}^{\varepsilon }}\subset
\Omega $ so that $\Sigma ^{\varepsilon }=\partial \Omega _{2}^{\varepsilon }$
and $\partial \Omega _{1}^{\varepsilon }=\partial \Omega \cup \Sigma
^{\varepsilon }$. Let $Z_{i}=\cup _{k\in \mathbb{Z}^{N}}\left(
Y_{i}+k\right) $. As in \cite{all}, we also assume that\ $Z_{1}$ is smooth
and a connected open subset of $\mathbb{R}^{N}$. Let $A_{1}$ (resp. $A_{2}$)
denote the permeability of the medium $Z_{1}$ (resp. $Z_{2}$). Let $f_{i}$
be a measurable function representing the thermal source density in the
material $\Omega _{i}^{\varepsilon }$. Finally, let $\gamma $ be the \textit{%
non-rescaled} conductivity of the thin layer $\Sigma ^{\varepsilon }$. We
shall assume the followings:

\begin{enumerate}
\item[H1)] The conductivity tensors $A_{1}$ and $A_{2}$ are continuous on $%
\mathbb{R}^{N}$, $Y-$periodic and satisfy the ellipticity condition:\textbf{%
\ }%
\begin{equation*}
A_{1}\xi \cdot \xi \geq C|\xi |^{2},\ \text{(resp. }A_{2}\xi \cdot \xi \geq
C|\xi |^{2}\text{)}\ \ \ \ \ \forall \xi \in \mathbb{R}^{N}
\end{equation*}%
where, here and in the rest of the paper, $C$ denotes any positive constant
independent of $\varepsilon $;

\item[H2)] The source terms $f_{1}$and $f_{2}$ lie in $L^{2}\left( \Omega
\right) $;

\item[H3)] The conductivity $\gamma $ is a continous function on $\mathbb{R}%
^{N}$, $Y-$periodic and bounded from below:%
\begin{equation*}
\gamma \left( y\right) \geq C>0,\ \ \ \ \text{ }\forall y\in \mathbb{R}^{N}.
\end{equation*}
\end{enumerate}

In the sequel, we shall denote for $x\in \mathbb{R}^{N}$,
\begin{equation*}
\chi _{i}^{\varepsilon }\left( x\right) =\chi _{i}\left( \frac{x}{%
\varepsilon }\right) ,\ \ A_{i}^{\varepsilon }\left( x\right) =A_{i}\left(
\frac{x}{\varepsilon }\right)
\end{equation*}%
and
\begin{equation*}
\gamma ^{\varepsilon }\left( x\right) =\varepsilon \gamma \left( \frac{x}{%
\varepsilon }\right) .
\end{equation*}%
In this paper, we shall study the multiscale modelling of the following set
of equations:
\begin{subequations}
\begin{align}
& -\mathrm{div}\left( A_{1}^{\varepsilon }\nabla u_{1}^{\varepsilon }\right)
=f_{1}\text{ in }\Omega _{1}^{\varepsilon }\text{,}  \label{c221b} \\
& -\mathrm{div}\left( A_{2}^{\varepsilon }\nabla u_{2}^{\varepsilon }\right)
=f_{2}\text{ in }\Omega _{2}^{\varepsilon }\text{,}  \label{c221a} \\
& A_{1}^{\varepsilon }\nabla u_{1}^{\varepsilon }\cdot \nu ^{\varepsilon
}-A_{2}^{\varepsilon }\nabla u_{2}^{\varepsilon }\cdot \nu ^{\varepsilon
}=\gamma ^{\varepsilon }\text{ on }\Sigma ^{\varepsilon }\text{,}
\label{c221c} \\
& u_{1}^{\varepsilon }=u_{2}^{\varepsilon }\text{ on }\Sigma ^{\varepsilon }%
\text{,}  \label{c221d} \\
& u^{\varepsilon }=0\text{ on }\partial \Omega \   \label{c221e}
\end{align}%
where $\nu ^{\varepsilon }$ stands for the unit normal of $\Sigma
^{\varepsilon }$ outward to $\Omega _{1}^{\varepsilon }$. Here, $\Omega
_{i}^{\varepsilon }$ represents the region with conductivity $%
A_{i}^{\varepsilon }$. In this connection, the quantitities $%
u_{1}^{\varepsilon }$ and $u_{2}^{\varepsilon }$ represent the temperature
in $\Omega _{1}^{\varepsilon }$ and $\Omega _{2}^{\varepsilon }$,
respectively. The equations (\ref{c221b}) and (\ref{c221a}) express the
steady heat diffusion of the temperature in $\Omega _{1}^{\varepsilon }$ and
$\Omega _{2}^{\varepsilon }$ respectively, with Fick's law $%
J_{i}^{\varepsilon }=A_{i}\nabla u_{i}^{\varepsilon }$ where $%
J_{i}^{\varepsilon }$ is the diffusion flux. Here we consider the problem
that describes steady heat diffusion processes in which at the local scale
there is a jump of diffusion fluxes with continuous temperature between the
two constituents due to the fact that the interface $\Sigma ^{\varepsilon }$
is an active membrane. Similar phenomena are also observed in chemical
reactions \cite{c22bimp}, \cite{c22cy}, \cite{c22pan}, continuum mechanics.
Finally, (\ref{c221e}) is the homogeneous Dirichlet condition on the
exterior boundary of $\Omega $. Non standard transmission conditions (\ref%
{c221c}) are well-studied, see \cite{c22mis, c22mis1, c22mis2, c22rc}. These
conditions modelizes for example the contact of two structures which are
glued together and in which one of those component is a thin adhesive layer
that is thermically or mechanically very different from the other one. In
mechanical engineering, such structures are commonly used in manufacturing
new materials. The numerical treatment of the corresponding mathematical
model is still intractable and requires the use of a high number of meshes
because of the presence of this layer, and thus leads to numerical
instability (see, e.g., Babuska and Suri \cite{c22bs}). Therefore, it is
more suitable from the numerical and modeling point of views to replace the
thin layer by an interface with zero thickness. This can be done for example
by performing an asymptotic method which takes into all the physical
properties of the layer, see for instance \cite{c22mis} and the references
therein.

Let us denote for convenience
\end{subequations}
\begin{equation*}
A^{\varepsilon }=\chi _{1}^{\varepsilon }A_{1}^{\varepsilon }+\chi
_{2}^{\varepsilon }A_{2}^{\varepsilon }\text{ and }f^{\varepsilon }=\chi
_{1}^{\varepsilon }f_{1}^{\varepsilon }+\chi _{2}^{\varepsilon
}f_{2}^{\varepsilon }.
\end{equation*}

The weak formulation of (\ref{c221b})-(\ref{c221e}) is as follows\textbf{: }%
find $\ u^{\varepsilon }=\chi _{1}^{\varepsilon }u_{1}^{\varepsilon }+\chi
_{2}^{\varepsilon }u_{2}^{\varepsilon }\in H_{0}^{1}\left( \Omega \right) $,
such that for all $v\in H_{0}^{1}\left( \Omega \right) $, we have
\begin{equation}
\left. \int_{\Omega }A^{\varepsilon }\nabla u^{\varepsilon }\nabla v\mathrm{d%
}x\ =\int_{\Omega }f^{\varepsilon }v\mathrm{d}x+\int_{\Sigma ^{\varepsilon
}}\gamma ^{\varepsilon }v\mathrm{d}s^{\varepsilon }\right.  \label{c22fv}
\end{equation}%
where $\mathrm{d}x$ and $\mathrm{d}s^{\varepsilon }$ denote respectively the
Lebesgue measure on $\mathbb{R}^{N}$ and the Hausdorff measure on $\Sigma
^{\varepsilon }$.

Now we give an existence and uniqueness result of (\ref{c22fv})

\begin{theorem}
Let assumptions H1)-H3) be fulfilled. Then, for any sufficiently small $%
\varepsilon >0$, there exists a unique couple $u^{\varepsilon }\in
H_{0}^{1}\left( \Omega \right) $, solution of the weak problem (\ref{c22fv}%
), such that
\begin{equation}
\Vert u^{\varepsilon }\Vert _{H_{0}^{1}\left( \Omega \right) }\leq C\text{.}
\label{c2230}
\end{equation}
\end{theorem}

\noindent \textbf{Proof:} \ It is a straightforward application of the Lax
Milgram Theorem. The only point to show is the continuity of the form
\begin{equation*}
v\longmapsto \int_{\Omega }f^{\varepsilon }v\mathrm{d}x+\int_{\Sigma
^{\varepsilon }}\gamma ^{\varepsilon }v\mathrm{d}s^{\varepsilon }
\end{equation*}%
on $H_{0}^{1}\left( \Omega \right) $. Indeed, let $v\in H_{0}^{1}\left(
\Omega \right) $. From the trace Theorem and by the rescaling technique (see
for example \cite{ep}). \hfill $\sqcap\!\!\!\!\sqcup$

Now, we are ready to state the main result of the paper:

\begin{theorem}
\label{c22thp}Let $u^{\varepsilon }\in H_{0}^{1}\left( \Omega \right) $ be
the unique solution of the weak system (\ref{c22fv}). Then, there exist a
unique $u\in H^{1}\left( \Omega \right) $ and a subsequence of $\left(
u^{\varepsilon }\right) $ still denoted $\left( w^{\varepsilon }\right) $
such that $w^{\varepsilon }$ converges weakly in $H^{1}\left( \Omega \right)
$ to $u$ which is the solution of the homogenized model:
\begin{equation}
\left\{
\begin{array}{l}
-\mathrm{div}\left( A^{\hom }\nabla u\right) =F\text{ in }\Omega , \\
\  \\
u=0\text{ on }\partial \Omega%
\end{array}%
\right.  \label{c22hs}
\end{equation}%
where $A^{\hom }$ and$\ F$ are given in (\ref{c22a}) and (\ref{c22b})
respectively.
\end{theorem}

The remainder of this paper is devoted to the proof of this Theorem. To
prove this result we shall employ the two-scale convergence technique.

\begin{definition}[G. Allaire \protect\cite{all}]
\label{c22def1}We say that a sequence $\left( v^{\varepsilon }\right) \ $in $%
L^{2}\left( \Omega \right) $ is two-scale convergent to $v_{0}\in
L^{2}\left( \Omega \times Y\right) $ (we write $v^{\varepsilon }\overset{2-s}%
{\rightharpoonup }v_{0}$) if, for any admissible test function $\varphi \in
L^{2}\left( \Omega ;\mathcal{C}_{\#}(Y)\right) $,
\begin{equation*}
\lim_{\varepsilon \rightarrow 0}\int_{\Omega }v^{\varepsilon }\left(
x\right) \varphi \left( x,\frac{x}{\varepsilon }\right)=\int_{\Omega \times
Y}v_{0}\left( x,y\right) \varphi \left( x,y\right)
\end{equation*}%
where $\mathcal{C}_{\#}(Y)$ is the space of all continuous functions on $%
\mathbb{R}^{N}$ which are $Y$-periodic.
\end{definition}

Let $L_{\#}^{2}\left( Y\right) $ be the space of all $Y$-periodic functions
belonging to $L_{\mathrm{loc}}^{2}\left( \mathbb{R}^{N}\right) $ and $%
H_{\#}^{1}\left( Y\right) $ denotes the space of those functions together
with their derivatives belonging to $L_{\#}^{2}\left( Y\right) $. The
following result concerns the two-scale convergence of uniformly bounded
sequences in the Sobolev space $H^{1}(\Omega )$. See \cite{all, adh}.

\begin{theorem}
\label{c22r1}\ Let $(v^{\varepsilon })$ be a uniformly bounded sequence in $%
H^{1}(\Omega )$ (resp. $H_{0}^{1}(\Omega )$). Then there exists $v_{0}\in
H^{1}\left( \Omega \right) $ (resp. $H_{0}^{1}(\Omega )$) and $v^{\ast }\in
L^{2}(\Omega ;H_{\#}^{1}(Y)/\mathbb{R})$ such that, up to a subsequence, $%
v^{\varepsilon }\overset{2-s}{\rightharpoonup }v_{0}$ and $\nabla
v^{\varepsilon }\overset{2-s}{\rightharpoonup }\nabla v_{0}+\nabla
_{y}v^{\ast }$. Furthermore for every $\varphi \in \mathcal{D}\left( \Omega ;%
\mathcal{C}_{\#}^{\infty }(Y)\right) $ we have
\begin{equation*}
\lim_{\varepsilon \rightarrow 0}\varepsilon \int_{\Sigma ^{\varepsilon
}}v^{\varepsilon }\varphi ^{\varepsilon }\mathrm{d}s^{\varepsilon
}=\int_{\Omega \times \Sigma }v_{0}\varphi \mathrm{d}x\mathrm{d}s
\end{equation*}%
where $\mathcal{C}_{\#}^{\infty }\left( Y\right) =\mathcal{C}^{\infty }(%
\mathbb{R}^{N})\cap \mathcal{C}_{\#}(Y)$ and $\varphi ^{\varepsilon }\left(
x\right) =\varphi \left( x,\frac{x}{\varepsilon }\right) $.
\end{theorem}

\section{Proof of Theorem \protect\ref{c22thp}\label{c22s3}}

In this section, we shall determine the limiting problem (\ref{c22hs}).
First, thanks to the a priori estimates(\ref{c2230}) and Theorem \ref{c22r1}%
, there exist a subsequence of $\left( u^{\varepsilon }\right) $, still
denoted $\left( u^{\varepsilon }\right) $ and a unique $u\in H_{0}^{1}\left(
\Omega \right) $,$\ u^{\ast }\in L^{2}(\Omega ;H_{\#}^{1}\left( Y\right) /%
\mathbb{R})$ such that $u^{\varepsilon }\overset{2-s}{\rightharpoonup }u$
and $\nabla u^{\varepsilon }\overset{2-s}{\rightharpoonup }\nabla u+\nabla
_{y}u^{\ast }$. Now, let $\varphi \in \mathcal{D}\left( \Omega \right) $ and
$\varphi ^{\ast }\in \mathcal{D}\left( \Omega ;\mathcal{C}_{\#}^{\infty
}\left( Y\right) \right) $. Set $\varphi ^{\varepsilon }(x)=\varphi \left(
x\right) +\varepsilon \varphi ^{\ast }\left( x,\dfrac{x}{\varepsilon }%
\right) $ and take $v=\varphi ^{\varepsilon }$ as a function test in in (\ref%
{c22fv}), we obtain
\begin{equation}
\left.
\begin{array}{c}
\int_{\Omega }A^{\varepsilon }\nabla u^{\varepsilon }\left( \nabla \varphi
+\nabla _{y}\varphi ^{\ast }\left( x,\dfrac{x}{\varepsilon }\right) \right)
\mathrm{d}x= \\
\int_{\Omega }f^{\varepsilon }\varphi \mathrm{d}x+\int_{\Sigma ^{\varepsilon
}}\gamma ^{\varepsilon }\varphi \mathrm{d}s^{\varepsilon }+\varepsilon
R^{\varepsilon }%
\end{array}
\right.  \label{c2246}
\end{equation}%
where
\begin{equation*}
R^{\varepsilon }=\int_{\Omega }A^{\varepsilon }\nabla u^{\varepsilon }\nabla
_{x}\varphi ^{\ast }\left( x,\frac{x}{\varepsilon }\right) \mathrm{d}%
x+\int_{\Sigma ^{\varepsilon }}\gamma ^{\varepsilon }\varphi ^{\ast }\left(
x,\frac{x}{\varepsilon }\right) \mathrm{d}s^{\varepsilon }\text{.}
\end{equation*}%
According to the assumptions H1)-H3) the vectorial functions $\chi
_{i}\left( ^{t}A_{i}\nabla \varphi \right) \ $\ and $\chi _{i}\left(
^{t}A_{i}\nabla \varphi ^{\ast }\right) $ $i=1,2$ are admissible in the
sense that it can be used as test functions w.r.t. the notion of two scale
convergence. It follows that the limit of the l.h.s. of (\ref{c2246}) is
\begin{equation*}
\int_{\Omega \times Y}A\left( \nabla u+\nabla _{y}u^{\ast }\right) \left(
\nabla \varphi +\nabla _{y}\varphi ^{\ast }\right) \mathrm{d}x\mathrm{d}y
\end{equation*}%
and that
\begin{eqnarray*}
&&\lim_{\varepsilon \rightarrow 0}\left( \int_{\Omega }f^{\varepsilon
}\left( x\right) \varphi \left( x\right) \mathrm{d}x+\int_{\Sigma
^{\varepsilon }}\gamma ^{\varepsilon }\varphi \left( x\right) \mathrm{d}%
s^{\varepsilon }\right) \\
&=&\int_{\Omega }F\left( x\right) \varphi \left( x\right) \mathrm{d}x,
\end{eqnarray*}%
where
\begin{equation*}
A\left( y\right) =\chi _{1}\left( y\right) A_{1}\left( y\right) +\chi
_{2}\left( y\right) A_{2}\left( y\right)
\end{equation*}%
and
\begin{equation}
F\left( x\right) =|Y_{1}|f_{1}\left( x\right) +|Y_{2}|f_{2}\left( x\right)
+\int_{\Sigma }\gamma \left( y\right) \mathrm{d}s\left( y\right) .
\label{c22b}
\end{equation}%
Moreover, using again (\ref{c2230}), it is easy to check that $%
R^{\varepsilon }=O\left( 1\right) $. Thus, by collecting all the above
limits we get the two-scale variational formulation:
\begin{equation}
\left. \int_{\Omega \times Y}A\left( \nabla u+\nabla _{y}u^{\ast }\right)
\left( \nabla \varphi +\nabla _{y}\varphi ^{\ast }\right) =\int_{\Omega
}F\varphi \text{.}\right.  \label{c2250}
\end{equation}%
By a density argument, equation (\ref{c2250}) still holds true for any $%
\left( \varphi ,\varphi ^{\ast }\right) \in H_{0}^{1}\left( \Omega \right)
\times L^{2}\left( \Omega ;H_{\#}^{1}\left( Y\right) /\mathbb{R}\right) $.
Now, integrating by parts in (\ref{c2250}) yields the following two-scale
homogenized system:
\begin{eqnarray}
&&-\mathrm{div}_{y}\left( A\left( \nabla u+\nabla _{y}u^{\ast }\right)
\right) \text{%
$=$%
\ }0\text{ in }\Omega \times Y,  \label{c22474} \\
&&-\mathrm{div}\left( \int_{Y}A\left( \nabla u+\nabla _{y}u^{\ast }\right)
\right) \text{%
$=$%
\ }F\text{ in }\Omega ,  \label{c22476} \\
&&u\text{%
$=$%
\ }0\text{ on }\partial \Omega ,  \label{c22478} \\
&&y\longmapsto \ u^{\ast }\text{ }Y\text{-periodic.}  \label{c22481}
\end{eqnarray}%
Let us denote for $1\leq j\leq N$, $\omega _{j}\in H_{\#}^{1}\left(
Y_{1}\right) /\mathbb{R}$ the unique solution to the following cell problem:
\begin{equation*}
\left\{
\begin{array}{l}
-\mathrm{div}_{y}\left( A\left( \nabla _{y}\omega _{j}+e_{j}\right) \right) =%
\text{\ }0\text{ in }Y_{1},\ \ \ \ \ \ \text{,} \\
A\left( \nabla _{y}\omega _{j}+e_{j}\right) \cdot \nu =0\text{ on }\Sigma ,
\\
y\longmapsto \omega _{j}\text{ }Y\text{-periodic,}%
\end{array}%
\right.
\end{equation*}%
where $\left( e_{j}\right) $ is the canonical basis of $\mathbb{R}^{N}$.
Observe that Equations (\ref{c22474}), (\ref{c22478}) and (\ref{c22481})
lead to the following relation:
\begin{equation}
u^{\ast }\left( x,y\right) =\sum_{j=1}^{N}\frac{\partial u}{\partial x_{j}}%
\left( x\right) \omega _{j}\left( y\right) +\tilde{u}\left( x\right)
\label{c223m}
\end{equation}%
where $\tilde{u}\left( x\right) $ is any additive function independent of $y$%
.\ In the sequel, we shall denote for convenience
\begin{equation}
\begin{array}{c}
A^{\hom }=\left( a_{ij}\right) _{1\leq i,j\leq N}, \\
a_{ij}=\int_{Y}A\left( \nabla _{y}\omega _{i}+e_{i}\right) \cdot \left(
\nabla _{y}\omega _{j}+e_{j}\right) \mathrm{d}y,%
\end{array}
\label{c22a}
\end{equation}%
Let us mention that, in view of H1), $A^{\hom }$ is symmetric and positive
definite, see \cite{blp}. Inserting (\ref{c223m}) into (\ref{c22476})
together with (\ref{c22478}) yields the elliptic equation:
\begin{eqnarray}
&&-\mathrm{div}\left( A^{\hom }\nabla u\right) \text{%
$=$%
}F\text{ in }\Omega ,  \label{c22hs1} \\
&&u\text{%
$=$%
}0\text{ on }\partial \Omega .  \label{c22hs2}
\end{eqnarray}%
The proof of Theorem \ref{c22thp} is then achieved.

\vspace{10pt} \noindent \textbf{Acknowledgements:} \ The paper is a part of
the project N$^\circ$ B00002278 of the MESRS algerian office. The author is
indebted for their financial support.


\begin{thebibliography}{99}
\bibitem{ain3} A.~Ainouz, Homogenization of a dual-permeability problem in
two-component media with imperfect contact, \emph{Applications~of~Mathematics%
}~60, 2015, pp.~185--196.

\vspace{-7pt}

\bibitem{all} G.~Allaire, Homogenization and two scale convergence, \emph{%
SIAM~J.~Math.~Anal.}~23, 1992, pp.~1482--1519.

\vspace{-7pt}

\bibitem{adh} G.~Allaire, A.~Damlamian, U.~Hornung, Two scale convergence on
periodic surfaces and applications, in: \emph{Proceedings of the
international conference on mathematical modelling of flow through porous
media (May 1995)} A.~Bourgeat \& \textit{al}. (Eds.)~World Scientific Pub.
Singapore 1996, pp.~15--25.

\vspace{-7pt}

\bibitem{adh1} T.~Arbogast, J.~Douglas and U.~Hornung, Derivation of the
double porosity model of single phase flow via homogenization theory, \emph{%
SIAM~J.~Math.~Anal.}~21, 1990, pp.~823--836.

\vspace{-7pt}

\bibitem{c22bs} I.~Babu\v{s}ka and M.~Suri, On locking and robustness in the
nite element method, \emph{SIAM~J.~Numer.~Anal.}~29, 1992, pp.~1261--1293 .

\vspace{-7pt}

\bibitem{c22bp} T.~Gnana Bhaskar and K.~Perera, On Some Elliptic Interface
Problems with Nonhomogeneous Jump Conditions, \emph{%
Advances~in~Nonlinear~Analysis}~2, 2013, pp.~195--211

\vspace{-7pt}

\bibitem{blp} A.~Bensoussan, J.--L.~Lions and G.~Papanicolaou, \emph{%
Asymptotic analysis for periodic structures}, North--Holland--Amsterdam 1978.

\vspace{-7pt}

\bibitem{c22bimp} K.~Bimpong--Bota, P.~Ortoleva and J.~Ross,
Far-From-equilibrium phenomena at local sites of reaction, \emph{%
J.~Chem.~Phys.}~60, 1974, pp.~3124--3133.

\vspace{-7pt}

\bibitem{c22cy} John M.~Chadam and Hong--Ming~Yin. A diffusion equation with
localized chemical reactions. \emph{Proc.~Edinburgh~Math.~Soc.}~2~37(1),
1994, pp.~101--118, .

\vspace{-7pt}

\bibitem{cs} G.~W.~Clark, R.~E.~Showalter, Two-scale convergence for a flow
in a partially fissured medium, \emph{Electr.~J.~of~Diff.~Equ.}~1999, 1999,
pp.~1--20.

\vspace{-7pt}

\bibitem{c22drm} M.~Dalla Riva, P.~Musolino, A Singularly Perturbed Nonideal
Transmission Problem and Application to the Effective Conductivity of a
Periodic Composite, \emph{SIAM~J.~Appl.~Math.}~73, 2013, pp.~24--46.

\vspace{-7pt}

\bibitem{c22drm2} M.~Dalla Riva, G.~Mishuris, Existence results for a
nonlinear transmission problem, \emph{J.~of~Math.~Anal.~and~Appl.}~430,
2015, pp.~718--741.

\vspace{-7pt}

\bibitem{ds} H.~Deresiewicz, R.~Skalak, On uniqueness in dynamic
poroelasticity, \emph{Bull.~Seismol.~Soc.~Amer.}~53, 1963, pp.~783--788.

\vspace{-7pt}

\bibitem{dm} P.~Donato, S.~Monsurr\'{o}, Homogenization of two heat
conductors with an interfacial contact resistance, \emph{%
Analysis~and~Applications}~2, 2004, pp.~247--273.

\vspace{-7pt}

\bibitem{ep} H.~I.~Ene, D.~Polisevski, Model of diffusion in partially
fissured media, \emph{Z.~angew.~Math.~Phys.}~53, 2002, pp.~1052--1059.

\vspace{-7pt}

\bibitem{c22jvs} S.~Jim\'{e}nez, B.~Vernescu, and W.~Sanguinet, Nonlinear
neutral inclusions: Assemblages of spheres, \emph{%
http:http://arxiv.org/abs/1201.4902v3}.

\vspace{-7pt}

\bibitem{c22mis} G.~Mishuris, W.~Miszuris, and A.~Ochsner, Evaluation of
transmission conditions for thin reactive heat-conducting interphases, \emph{%
Deffect~Diffus.~Forum}~273--276, 2008, pp.~394--399.

\vspace{-7pt}

\bibitem{c22mis1} G.~Mishuris, W.~Miszuris, and A.~Ochsner, Transmission
conditions for thin reactive heat conducting interphases: general case,
\emph{Deffect~Diffus.~Forum}~283--286, 2009, pp.~521--526.

\vspace{-7pt}

\bibitem{c22mis2} W.~Miszuris and A.~Ochsner, Universal transmission
conditions for thin reactive heat-conducting interphases, \emph{%
Continuum~Mechanics~and~Thermodynamics}~25, 2013, pp.~1--21 .

\vspace{-7pt}

\bibitem{c22pan} R.~Pan. A class of diffusion equations with localized
chemical reactions. \emph{Nonlinear~Anal.}~27, 1996, pp.~653--668,.

\vspace{-7pt}

\bibitem{c22rc} F.~Rosselli and P.~Carbutt, Structural bonding applications
for the transportation industry, \emph{SAMPE~J.}~37, 2001, pp.~7--13.
\end{thebibliography}
\end{document}